\documentclass[11pt]{amsart}
\usepackage{amssymb}
\usepackage{graphicx}
\usepackage{amscd}
\usepackage{amsmath}
\usepackage{amsmath}

\usepackage{tikz-cd}
\usepackage{amsfonts,latexsym,amstext}

\usepackage [T1]{fontenc}
\usepackage [latin1]{inputenc}
\usepackage{color}
\usepackage[all]{xy}
\usepackage{cancel}

\newtheorem{theorem}{Theorem}[section]
\newtheorem{definition}[theorem]{Definition}

\newtheorem{lemma}[theorem]{Lemma}
\newtheorem{proposition}[theorem]{Proposition}
\newtheorem{corollary}[theorem]{Corollary}
\theoremstyle{remark}

\newtheorem{remark}[theorem]{Remark}

\newcommand\addtag{\refstepcounter{equation}\tag{\theequation}}

\newcommand{\spaces}[3]{\ensuremath{#1_{#2}, \ldots , #1_{#3}}} %Spaces  1_2,...,1_3
\newcommand{\cartesian}[3]{\ensuremath{#1_{#2} \times \cdots \times #1_{#3}}} %Cartesian product  1_2x...x1_3
\newcommand{\Tensor}[4]{\ensuremath{#1_{#2}}{\otimes}_{#4}\cdots{\otimes}_{#4}#1_{#3}} %1_2 otimes_3...otimes_3 1_4    alpha-tensor product of n spaces.
\newcommand{\CTensor}[4]{\ensuremath{#1_{#2}}{\hat{\otimes}}_{#4}\cdots{\hat{\otimes}}_{#4}#1_{#3}}
\usepackage{tikz-cd}
\usetikzlibrary{arrows}

\tikzset{
commutative diagrams/.cd,
arrow style=tikz,
diagrams={>=latex}}

\begin{document}

\title{The Segre cone of Banach spaces and multilinear operators}

\author{Maite Fern\'{a}ndez-Unzueta}

%\thanks{The  author was partially supported by  CONACYT 182296}

%\subjclass[2010]{46B28, 47H60, 47L22, and  15A03}

\address{Centro de Investigaci\'{o}n en Matem\'{a}ticas (Cimat), A.P. 402 Guanajuato, Gto., M\'{e}xico} \email{maite@cimat.mx}

%\subjclass[2010]{}

\begin{abstract} We prove that any pair of    reasonable cross norms  defined on
the tensor product  of $n$ Banach spaces  induce $(2k)^{n-1}$-Lipschitz equivalent metrics (and thus, a unique topology)  on  the set $S^k_{X_1,\ldots, X_n}$ of vectors of rank  $\leq k$. With this, we  define the Segre cone of Banach spaces, $\Sigma_{X_1,\ldots, X_n},$ and  state when $S^k_{X_1,\ldots, X_n}$ is  closed. We introduce an auxiliary mapping (a $\Sigma$-operator) that allows us to study  multilinear mappings with a  geometrical point of view. We use  the isometric correspondence between 
multilinear mappings  and Lipschitz $\Sigma$-operators,  to have a strategy to generalize ideal properties from the linear to the multilinear settitng. 
\end{abstract}
\maketitle

\section{Introduction}\label{sec: intro}

    Many different efforts  have  been made to extend the  well known theory of ideals of linear operators on Banach spaces to the context of multilinear mappings. Among others, the following ideals had been generalized:  compact and  weakly compact operators \cite{AronGalindo97};   absolutely $p$-summing operators \cite{Pietsch83} (other generalizations of this ideal may be found in  \cite{A FU} and \cite{Calis Pelle} and  references therein);  unconditionally converging operators \cite {BomFUnVil} or Hilbert-Schmidt operators \cite{Matos03}.

    In this paper we  tackle  the  problem of generalizing from the  linear to the multilinear theory,  by means of a   geometrical tool  that  we define as   $\Sigma$-operators. These are mappings that make visible      the geometric richness occurring  where the metric, the tensor and the multilinear structures merge.  The  strategy proposed to generalize the linear theory  begins by factorizing   a multilinear mapping  between vector spaces, $T\in {L}(\spaces{X}{1}{n};Y)$,   in the following way:
     \[
\begin{tikzcd}\label{diagram}
 & X_1\times\ldots \times X_n \arrow{ld}[swap]{\otimes}
                                \arrow{rd}{T} &\\
    \Sigma_{X_1,\ldots, X_n}  \arrow{rr}{f_T}
                                 \arrow[hook]{rd}{i}   & & Y \\
              & \Tensor{X}{1}{n}{} \arrow{ru}{\hat{T} }   & \addtag
\end{tikzcd}
\]
     where ${\Sigma}_{X_1,\ldots,X_n}:=\{\Tensor{x}{1}{n}{}\in  \Tensor{X}{1}{n}{}; \; x_i\in X_i\}
   $ denotes  the set of decomposable tensors, $\hat{T}$ is  the   linear mapping associated to ${T}$ and $f_T:=\hat{T}_{|_{{\Sigma}_{X_1,\ldots,X_n}}}$.
   Essentially, the proposal  will consist in  studying   $f_T$  ($\Sigma$-operators,  Defnition \ref{def: sigma-oper}) instead of $T$,  in the Banach space category.
   This paradigm shift  makes us  gain    access to
  geometric features  of  the domain, ${\Sigma}_{X_1,\ldots,X_n}$,  that  remain   hidden  in the   multilinear  expression, whose   domain,  $\cartesian{X}{1}{n}$, has  much a poorer structure.
   The paper is organized in two parts: in Section \ref{sec: Tensors Sk} we carry    out  an accurate study of  tensor metrics  on the set of decomposable tensors,  ${\Sigma}_{X_1,\ldots,X_n}$. This allows such  domains  to be  well-defined.  $\Sigma$-operators and the way  the proposal works  are established   in Section \ref{sect: morphisms}.

       The main result   in Section \ref{sec: Tensors Sk} is Theorem \ref{thm: fixed rank  cone metrics}. There  we prove that for a fixed  $k$ and for any pair of reasonable cross norms,  $\alpha, \beta$, defined on
the tensor product $\Tensor{X}{1}{n}{}$ of Banach spaces, the metrics  induced on $ \mathcal{S}^k_{\spaces{X}{1}{n}}$, the set of tensors of range $\leq k$,  are $(2k)^{n-1}$-Lipschitz equivalent.  These constants provide an asymptotic estimate of the discrepancy between $\alpha$ and $\beta$  on the completed tensor spaces. A consequence of Theorem \ref{thm: fixed rank  cone metrics} is that   all reasonable cross norms induce a  unique  topology on  $ \mathcal{S}^k_{\spaces{X}{1}{n}}$. With this,  we  define de  the {\sl Segre cone} of $n$-spaces in the category of Banach spaces (Definition \ref{def: Sigma}), and settle  its basic geometrical properties. A multilinear Auerbach's Theorem, as well as the estimate of the Banach-Mazur distance
 $ d_{BM}(\Tensor{F}{1}{n}{\pi}, \Tensor{F}{1}{n}{\epsilon})\leq \Pi_{j=1}^{n-1}d_j,$ in terms of the dimensions $d_j$ of the subspces $F_j$, are given.

   We finish this section proving  that  the Segre  cone,  $\Sigma_{X_1,\ldots, X_n}$, and   $S^k_{X_1, X_2}$  are  complete metric spaces (Propositions \ref{prop: cSegre cerrado en tensor} and  \ref{prop: Sr Hilbert spaces}, resp.), while  $S^k_{X_1,\ldots, X_n}$  is not, if  $k\geq 2$ and   $n\geq 3$    (Proposition \ref{prop: Sr not closed}). These  results are related to  the so  called low rank approximation problem which, in the finite dimensional case, has been deeply studied in  \cite{DSilva Lim}.  In  \cite{Comon}, \cite{Land08} and  \cite{Land} one may  find    tensor decomposition problems  treated   both from a theoretical point of view and in applications, arising from varied sources as computer science,    approximation theory   or  algebraic complexity theory.

 In Section \ref{sect: morphisms} we outline our  general approach to study  multilinear mappings: it is based on Theorem  \ref{thm: continuity diagram}, where  we prove that a multilinear mapping $T$ is bounded if and only if its associated $\Sigma$-operator, $f_T$, (Definition \ref{def: sigma-oper})  is Lipschitz on the Segre cone. Thus, an  analogous  diagram to (\ref{diagram}) holds   in the Banach space category. In light of this general scheme,  the generalization process we propose is as follows:

 Given a  boundedness  condition on linear  operators,$\{S:X\rightarrow Y\}$,  it is possible to write    the analogous Lipschitz condition on $\Sigma$-operators,  $\{f:\Sigma_{\spaces{X}{1}{n}}\rightarrow Y\}$.
Finally, the relation $T(\spaces{x}{1}{n})=f_T(\Tensor{x}{1}{n}{})$ and diagram (\ref{diagram}) give rise, naturally, to the desired  generalized condition on multilinear operators.

To exemplify this,  in the last Section we present  the case of   $p$-summabilibty  developed in the forthcomming paper \cite{A FU}.
 It is worth mentioning that this  approach gives rise to a new notion of multilinear $p$-summability   for which    the three fundamental equivalences of the notion of a $p$-summing linear operator (namely, its local definition,  a domination by an $L_p$-norm, and a factorization through a subset of an $L_p$-space) hold in a natural way.

 \section{Tensors of rank $\leq k$ in Banach spaces }\label{sec: Tensors Sk}

 The study of  {\sl tensors} (elements in a tensor product space, \cite[Definition 2.3.1.3]{Land}) frequently entails   the study of  an associated topological structure. Thus, for example, the Zariski topology is inherent to the study of the secant varieties of the Segre variety (the projectivization of the Segre cone).  In like manner,  a  normed  structure is assumed on the spaces when  dealing  with tensor approximation problems, as is the case of the Frobenius norm in the celebrated Eckart-Young's Theorem on approximation of a matrix by lower rank matrices \cite{EckartYoung}.

  When the spaces are finite dimensional, all norms  induce  the same  topology. As noted in \cite[Section 4]{DSilva Lim}, this property allows,  in many cases,   to choose an appropriate norm to work with.

In the case of infinite dimensional spaces  it is no longer  possible to argue in this way, since  it is possible to define   infinitely many non-equivalent norms on a vector space.

In this section we prove that,  in spite of this, the closure of tensors of a fixed rank remains  independent of the norm chosen, provided it  is a reasonable cross norm.

\

Let   $X_1,\ldots,X_n$ be normed spaces  over $\mathbb{R}$ or  $\mathbb{C}$. A norm $\alpha$ on the vector space $\Tensor{X}{1}{n}{}$ is said to be a {\bf reasonable cross norm} if it has the following properties:
  \begin{enumerate} \item $\alpha(x_1\otimes\dots\otimes x_n)\leq \|x_1\|\cdots \|x_n\|$ for every $x_i\in X_i;\; i=1,\ldots n. $
  \item For every $x_i^*\in X_i^*$, the linear functional $x_1^*\otimes\dots\otimes x_n^*$ on $\Tensor{X}{1}{n}{}$ is bounded, and $\|x_1^* \otimes\dots\otimes x_n^*\|\leq \|x_1^*\|\cdots \|x_n^*\|.$
  \end{enumerate}

   A norm $\alpha$ on the vector space $\Tensor{X}{1}{n}{}$ is a reasonable cross norm if and only if for  every  $ z\in \Tensor{X}{1}{n}{}$
    $$ \epsilon(z)\leq \alpha(z)\leq \pi(z),$$
  where $\epsilon$ and $\pi$ are, respectively,  the injective and projective norms on   $\Tensor{X}{1}{n}{}$ (for the original proof, see \cite[Theorem 1 p.8]{Groth}; see also \cite[Proposition 6.1]{Ryan-libro}, \cite[Theorem 1.1.3]{DieFouSwa}). Thus,  inequalities   (1) and (2) are  equalities.

  The  normed space determined by a reasonable cross norm  $\alpha$ on the space  $\Tensor{X}{1}{n}{}$  will be  denoted by $ \Tensor{X}{1}{n}{\alpha}  $ and its completition by    $ \CTensor{X}{1}{n}{\alpha}$.

\

  Given    vector  spaces,  $X_1,\ldots,X_n$,   the  {\sl rank}  of a tensor  $z\in \Tensor{X}{1}{n}{}$, is defined as:
       \begin{equation*}
        r_z:=\min \{r\in \mathbb{N}; \; z=\sum_{i=1}^r \Tensor{x^i}{1}{n}{};\;  x_i \in X_i;\; i=1,\ldots,n\}.
       \end{equation*}
       Any expression of the form  $z=\sum_{i=1}^r \Tensor{x^i}{1}{n}{}$ is called a \sl{ tensor decomposition of $z$}. We say that it is \sl{minimal}  when  $r=r_z$.
   For $r\in\mathbb{N}$,  we will  denote
      $$\mathcal{S}^r_{X_1,\ldots,X_n}:=\{z\in  \Tensor{X}{1}{n}{};\,\, r_z\leq r\}. $$

     The set  $\mathcal{S}^1_{X_1,\ldots,X_n}$ is the  \sl{Segre cone} of the spaces  $X_1,\ldots,X_n$, described in (\ref{diagram}).

 \hspace{1cm}

      From now on, the spaces    $\spaces{X}{1}{n}$ will be assumed to be Banach spaces.  If $\alpha$ is  a reasonable cross  norm  on $\Tensor{X}{1}{n}{}$, we will say that the rank of $z$ is infinite when  $z\in \CTensor{X}{1}{n}{\alpha}\setminus  \Tensor{X}{1}{n}{}$.

   For each fixed  $r\in \mathbb{N}$,
    $d_{\alpha}$  will denote   the  metric induced by the cross norm $\alpha$, on  the subset    $\mathcal{S}^r_{\spaces{X}{1}{n}}$ of tensors of rank $\leq{r}$. We will refer to these metrics as {\bf tensor metrics}.

  \begin{theorem}\label{thm: fixed rank  cone metrics} Let  $\spaces{X}{1}{n}$ be Banach  spaces, $r \in \mathbb{N}$   and   $\alpha$ and  $\beta$ reasonable cross norms  on  $\Tensor{X}{1}{n}{}$.
     Then, the following  metric spaces are Lipschitz equivalent:   $$(\mathcal{S}^r_{\spaces{X}{1}{n}}, d_{\alpha}) \hspace{.2cm}  \stackrel{Id}{\simeq}  \hspace{.2cm}  (\mathcal{S}^r_{\spaces{X}{1}{n}}, d_{\beta}). $$  In fact,
 for every $w,z \in \mathcal{S}^r_{X_1,\ldots,X_n}$
  $ d_{\alpha}(w,z)\leq (2r)^{n-1} d_{\beta}(w,z)$ and
 { $\|z\|_{{\alpha}}\leq r^{n-1}\|z\|_{{\beta}}. $}
   \end{theorem}

\

\begin{remark}\label{rmk: not optimal bound} The Lipschitz constant  $(2r)^{n-1}$ given in Theorem \ref{thm: fixed rank  cone metrics}  is not optimal. For example, for   any cross norm $\alpha$ and any $r$,
  $d_{{\epsilon}}(w,z)\leq  d_{\alpha}(w,z)\leq d_{{\pi}}(w,z)$.  In the particular case of the injective and the projective tensor norms, we have that for   a $z\in \mathcal{S}^r_{X_1,\ldots,X_n}$,
$$ d_{\epsilon }(z,w) \leq d_{\pi }(z,w) \leq (2r)^{n-1}d_{\epsilon }(z,w) \;\;\;\;\mbox{and}\;\;\;\; \|z\|_{\epsilon }\leq \|z\|_{\pi}\leq r^{n-1} \|z\|_{\epsilon }.$$
\end{remark}

We will give the proof of  the theorem at the end of the section.  Before,   we will  prove several  related results.
\

\begin{corollary}\label{coro: completeness equiv all tensor norm}   The closure  of
$\mathcal{S}^r_{\spaces{X}{1}{n}}$ in $ \CTensor{X}{1}{n}{\alpha}$ coincides with its  closure   in $ \CTensor{X}{1}{n}{\pi}$.  In particular,  $(\mathcal{S}^r_{\spaces{X}{1}{n}}, d_{\alpha})$  is complete  if and only if $(\mathcal{S}^r_{\spaces{X}{1}{n}}, d_{\beta})$ is complete.

\end{corollary}

\begin{proof}[Proof of Corollary]
Here, $\alpha, \beta$ are as in Theorem \ref{thm: fixed rank  cone metrics}.  Let  $i: \CTensor{X}{1}{n}{\pi}\longrightarrow \CTensor{X}{1}{n}{\alpha}$   be the inclusion mapping.  Observe that $i_{|_{\mathcal{S}^r}}:(\mathcal{S}^r,d_{\pi})\longrightarrow (\mathcal{S}^r,d_{\alpha})$. We denote by  $\nu_{\pi}$ and $\nu_{\alpha}$   the    inclusion mappings into  the  respective completed spaces. Let us check that the mappings in the following diagrams are well defined:

      \[
\begin{tikzcd}[row sep=2em,column sep=4em]
    (\mathcal{S}^r,d_{\pi}) \arrow{r}{i_{|_{\mathcal{S}^r}}}\arrow{d}{\nu_{\pi}}  & (\mathcal{S}^r,d_{\alpha}) \arrow{d}{\nu_{\alpha}} &  (\mathcal{S}^r,d_{\alpha}) \arrow{r}{(i_{|_{\mathcal{S}^r}})^{-1}}\arrow{d}{\nu_{\alpha}}  & (\mathcal{S}^r,d_{\pi}) \arrow{d}{\nu_{\pi}}  \\
    \overline{(\mathcal{S}^r,d_{\pi})} \arrow[dotted]{r}{\mathcal{I}}  &  \overline{(\mathcal{S}^r,d_{\alpha})} & \overline{(\mathcal{S}^r,d_{\alpha})} \arrow[dotted]{r}{\mathcal{J}} &  \overline{(\mathcal{S}^r,d_{\pi})} \\
\end{tikzcd}
\]
Each mapping  in the first  commutative diagram  has Lipschitz norm $\leq 1$. The dotted arrow ${\mathcal{I}}$  is  the extension of $ \nu_{\alpha} \circ (i_{|_{\mathcal{S}^r}})$  to the completed space   $\overline{(\mathcal{S}^r,d_{\pi})}$.
 By Theorem \ref{thm: fixed rank  cone metrics}  we have that $(i_{|_{\mathcal{S}^r}})^{-1}$, and therefore  $\nu_{\pi} \circ (i_{|_{\mathcal{S}^r}})^{-1}$,  has Lipschitz norm $\leq (2r)^{n-1}$. Consequently, the same constant bounds the Lipschitz norm of
 its  extension  to the closure, which is
  the dotted arrow ${\mathcal{J}}$.

  The mappings  $Id_{\pi}: {{\overline{(\mathcal{S}^r,d_{\pi})}}}\rightarrow \overline{(\mathcal{S}^r,d_{\pi})}$ and $\mathcal{J}\circ \mathcal{I} $ are equal, since they coincide  in the   dense subset $(\mathcal{S}^r,d_{\pi})$. An   analogous argument shows that   $Id_{\alpha}=\mathcal{I}\circ \mathcal{J}$, and the result follows.

Of course, if  $\beta$  is any other   reasonable cross norms, the closures  must also coincide:
 $\overline{(\mathcal{S}^r, \delta_{\alpha})}= \overline{(\mathcal{S}^r , \delta_{\pi})}=\overline{(\mathcal{S}^r, \delta_{\beta})}$.
  The assertion on the completeness is immediate from this.
\end{proof}

 \begin{proposition}\label{prop: tensor subspaces} Let $\spaces{X}{1}{n}$ be Banach spaces and, for each $j=1, \ldots , n$, let $F_j\subset X_j$ be a subspace of finite dimension $d_j$.  Then, for every $z\in \Tensor{F}{1}{n}{},$
 $$\|z\|_{\CTensor{F}{1}{n}{\pi}} \leq d_1\cdots d_{n-1}\|z\|_{\CTensor{F}{1}{n}{\epsilon}}.  $$ Consequently,  $\|z\|_{\CTensor{X}{1}{n}{\pi}} \leq d_1\cdots d_{n-1}\|z\|_{\CTensor{X}{1}{n}{\epsilon}}.$
 \end{proposition}

 \begin{proof}
  For each $j=1,\ldots,n-1$, let  $T_j$  be a projection mapping  onto $F_j$. There exist  representations of the form  $T_j=\sum_{i=1}^{d_j}  f_{i}^j\otimes w_{i}^j$ with nuclear norm  $\nu(T_j)=d_j = \sum_{j=i}^{d_j} \cdot \|f_{i}^j\|\cdot  \|w_{i}^j\|$ where $w_{i}^j\in F_j$ and $\|f_j\|_{X_j^*}=\|{f_{i}^j}_{|_{F_j}}\|=1$ (see \cite[Theorem 2]{Ruston}).
  Let $z= \sum_{i=1}^{m}{z}_i^1 \otimes\cdots \otimes {z}_{i}^n\in \Tensor{F}{1}{n}{}$.  Then, $z_i^j=T_j(z_i^j)$ and
 \begin{align*}
 \|z\|&_{\CTensor{F}{1}{n}{\pi}}   = \|\sum_{i=1}^{m}T_1({z}_i^1) \otimes\cdots \otimes T_{n-1}({z}_i^{n-1})\otimes {z}_{i}^n\|_{{\CTensor{F}{1}{n}{\pi}}}   \\
   & = \|\sum_{i=1}^{m}(\sum_{k_1=1}^{d_1}f_{k_1}^1({z}_i^1)w_{k_1}^1) \otimes\cdots \otimes   (\sum_{k_{n-1}=1}^{d_{n-1}}f_{k_{n-1}}^1({z}_i^{n-1})w_{k_{n-1}}^{n-1})\otimes {z}_{i}^n\|_{\CTensor{F}{1}{n}{\pi}}  \\
   & {\leq} \sum_{k_1,\ldots, k_{n-1} } \|w_{k_1}^1\| \cdots \|w_{k_{n-1}}^{n-1}\| \cdot\| \sum_{i=1}^{m}f_{k_1}^1({z}_i^1)\cdots f_{k_{n-1}}^{n-1}({z}_i^{n-1})  {z}_{i}^n\| \\
   & \stackrel{*}{\leq}  \sum_{k_1,\ldots, k_{n-1} } \|w_{k_1}^1\| \cdots \|w_{k_{n-1}}^{n-1}\| \|f_{k_1}^1\| \cdots \|f_{k_{n-1}}^{n-1}\| \cdot\| \sum_{i=1}^{m}{z}_i^1\otimes\cdots \otimes   {z}_{i}^n\|_{\CTensor{X}{1}{n}{\epsilon}} \\   & {\leq} d_1\cdots d_{n-1}\|\sum_{i=1}^{m}{z}_i^1 \otimes\cdots \otimes {z}_{i}^n \|_{{\CTensor{X}{1}{n}{\epsilon}}} \leq d_1\cdots d_{n-1}\|\sum_{i=1}^{m}{z}_i^1 \otimes\cdots \otimes {z}_{i}^n \|_{{\CTensor{F}{1}{n}{\epsilon}}}
  \end{align*}
 Inequality (*) holds because the injective tensor norm satisfies that for $z=\sum_{i=1}^l z_i^1\otimes\cdots\otimes z_i^n$ and  $L_z\in \mathcal{L}(\cartesian{X^*}{1}{n-1}; X_n) $ defined as  $L_z(\spaces{\phi}{1}{n-1}):= \sum_{i=1}^l\phi_1(z_i^1)\otimes\cdots\otimes \phi_{n-1}(z_i^{n-1})z_i^n$,  it holds that $\|z\|_{\epsilon}=\|{L}_z\|$ .

   This, along with the injectivity of the $\epsilon$-norm, proves the first statement. The second one follows from this and  the fact that    $ \|z\|_{\CTensor{X}{1}{n}{\pi}} \leq\|z\|_{\CTensor{F}{1}{n}{\pi}}$ and   $\|z\|_{\CTensor{X}{1}{n}{\epsilon}} =\|z\|_{\CTensor{F}{1}{n}{\epsilon}}$ always hold.
 \end{proof}

   The Banach-Mazur distance  between isomorphic Banach spaces    (see \cite{TomczJae-libro})  is  $$d_{BM}(X,Y)=\inf\{\|T\|\,\|T^{-1}\|;\; T\in \mathcal{L}(X,Y) \;  \mbox {is an onto isomorphism }\} $$

  \begin{corollary}\label{coro: BMazur} Let $F_1,\ldots, F_n$, be normed spaces of finite dimensions $d_1,\ldots, d_n$. Then:
 \begin{gather}\label{gath: Banach-Mazur} d_{BM}(\Tensor{F}{1}{n}{\pi}, \Tensor{F}{1}{n}{\epsilon})\leq \Pi_{j=1}^{n-1}d_j.
 \end{gather}
  \end{corollary}

  It is important to note that the bound in the corollary is obtained with the Id isomorphism.
  Also observe  that  the best bound is achieved when   the spaces are  reordered in such a way that    $d_j\leq d_n$, $j=1,\ldots n-1$.  Thus, for example, if $N\to \infty$, the following  distances remain bounded:
  $$ d_{BM}(\Tensor{{\ell_d^2}}{}{}{\pi}\otimes_{\pi}\ell_N^2, \Tensor{{\ell_d^2}}{}{}{\epsilon}\otimes_{\epsilon}\ell_N^2) \leq d^{n-1}.$$

 \

 \

 \begin{proof}{(of Theorem \ref{thm: fixed rank  cone metrics})}
  Consider $w,z \in \mathcal{S}^r_{X_1,\ldots,X_n}$. Then, there exist expressions $w=\sum_{i=1}^r  {w}_{i}^1 \otimes\cdots\otimes {w}_{i}^n$  and  $z=\sum_{i=1}^r  {z}_{i}^1 \otimes\cdots\otimes {z}_{i}^n$
with ${w}_{i}^j, {z}_{i}^j \in X_j$. For each $j=1,\ldots,n$,   consider  the  finite dimensional subspace $F_j=span\{w_1^j,\ldots,w_r^j, z_1^j,\ldots,z_r^j\}\subset X_j;  $ $\dim(F_j):=d_j\leq 2r$. Using  Proposition \ref{prop: tensor subspaces}, we  get
\begin{align*}
 d_{\alpha}(w,z)=& \|w-z\|_{\CTensor{X}{1}{n}{\alpha}}\leq   \|w-z\|_{\CTensor{X}{1}{n}{\pi}}\leq \\
    & (2r)^{n-1}\|w-z\|_{\CTensor{X}{1}{n}{\epsilon}}\leq(2r)^{n-1}\|w-z\|_{\CTensor{X}{1}{n}{\beta}}=(2r)^{n-1} d_{\beta}(w,z).
\end{align*}
If we interchange the  roles of $\alpha$ and $\beta$ in the previous estimation, we obtain  the assertion that
the metrics $d_{\alpha}$ and $d_{\beta}$ are Lipschitz  equivalent on $\mathcal{S}^r_{X_1,\ldots,X_n}$.

Note that, when  $w=0$, the subspace $F_j=\{ z_1^j,\ldots,z_r^j\}$ has dimension $d_j\leq r$. In this case, Propositon \ref{prop: tensor subspaces} establishes the inequality  $\|z\|_{{\alpha}}\leq r^{n-1}\|z\|_{{\beta}}. $

\end{proof}

\subsection*{A multilinear  Auerbach's Theorem}

\

  The  existence of {\sl Auerbach bases} is behind the proof of Proposition \ref{prop: tensor subspaces}.
Recall that
a basis $ \{\spaces{x}{1}{d}\}$ in a  finite dimensional Banach space $F$, is called an {\sl Auerbach basis  of $F$} if it is normalized and there exists a normalized basis of the dual space,  $ \{\spaces{x^*}{1}{d}\}\subset F^*$, such that  $ x^*_i(x_j)=\delta_j^i$ and $\|x_i\|=\|x_i^* \|=1$ (see  \cite[Proposition 1.c.3]{LTz}).
  This geometric tool can be constructed   in tensor product spaces    as follows:

 \begin{lemma}[{Multilinear Auerbach Lemma}]\label{lemma: Auerbach basis}
  Let $\spaces{F}{1}{n}$ be finite dimensional Banach spaces.  If  for each $j=\{1,\ldots, n\}$, $ \{\spaces{x^j}{1}{d_j}\}$ is  an Auerbach's basis of $F_j$, then,  for any  reasonable cross norm   $\alpha$ be a $\Tensor{F}{1}{n}{}$,
   the set $\beta:=\{x_{k_1}^1\otimes\cdots\otimes x_{k_n}^n\}_{\substack{k_1=1,\ldots,d_1 \\  \hspace{.4cm} \ldots \\  k_n=1,\ldots,d_n }}$ is an Auerbach basis of $\Tensor{F}{1}{n}{\alpha}$, whose dual basis is $\{{x^*}_{k_1}^1\otimes\cdots\otimes {x^*}_{k_n}^n\}_{\substack{k_1=1,\ldots,d_1 \\ \hspace{.4cm}\ldots \\  k_n=1,\ldots,d_n }}$.
  \end{lemma}
  \begin{proof} It is clear that $\beta$ is a basis of the vector space $\Tensor{F}{1}{n}{}$.  Let   $ \{\spaces{{x^*}^j}{1}{d_j}\}$ denote  the biorthogonal Auerbach basis in   each $F^*_j$.
   The system  $\beta^*:=\{{x^*}_{k_1}^1\otimes\cdots\otimes {x^*}_{k_n}^n\}_{\substack{k_1=1,\ldots,d_1 \\ \ldots \\  k_n=1,\ldots,d_n }}$ is   biorthogonal to  $\beta$. This follows from the relation
     $ ({x^*}_{k_1}^1\otimes\cdots\otimes {x^*}_{k_n}^n)(x_{l_1}^1\otimes\cdots\otimes x_{l_n}^n):={x^*}_{k_1}^1({x}_{l_1}^1)\cdots {x^*}_{k_n}({x}_{l_n}^n).$  It only remains to check   that every element in  this system  has norm $1$: since    $\alpha$ is a cross norm,   we have that $\alpha(x_{k_1}^1\otimes\cdots\otimes x_{k_n}^n)=\|x_{k_1}^1\|\cdots  \| {x}_{k_n}^n\|=1$ and $\|{x^*}_{k_1}^1\otimes\cdots\otimes {x^*}_{k_n}^n\|=\|{x^*}_{k_1}^1\|\cdots \| {x^*}_{k_n}^n\|=1$.
 \end{proof}

At the price of a worse constant, $d_1\cdots d_n$, it is possible to prove  Proposition \ref{prop: tensor subspaces} using the  Multilinear Auerbach lemma,  in a more  direct and geometric way.

\

\subsection*{Asymptotic comparison between  $d_{\alpha}$ and $  d_{\beta}$ in   $S^r_{X_1,\ldots, X_n}$}

Although all  cross norms are  equivalent to each other on   vectors of rank $\leq k$, in general  they are not  uniformly equivalent,  varying $k$.   The following constants capture    the asymptotic discrepancy between two reasonable cross norms:
\begin{align}\label{eq: def sigma_k}
    \sigma_{\spaces{X}{1}{n}}^r(\alpha, \beta):=  inf \{C>0; \;\;\;d_{\alpha}(w,z)\leq C \; d_{\beta}(w,z) \;\;\; \mbox{for any}\;\;\;  w, z \in \mathcal{S}_{{\spaces{X}{1}{n}}}^r\}. \nonumber
   \end{align}

 \begin{proposition}\label{prop: gap constants} Let $\spaces{X}{1}{n}$, ${\alpha}$ and ${\beta}$ as in Theorem \ref{thm: fixed rank  cone metrics}.
  Then, the  multilinear mapping $\otimes: \cartesian{X}{1}{n}\rightarrow \CTensor{X}{1}{n}{\alpha}$ determines a  continuous bounded linear operator
   $\Psi: \CTensor{X}{1}{n}{\beta}\rightarrow \CTensor{X}{1}{n}{\alpha}$  if, and only if,  the sequence
  $ \{   \sigma_{\spaces{X}{1}{n}}^r(\alpha, \beta)\}_r  $
  is  bounded.
  \end{proposition}
  \begin{proof} Observe that $\Psi_{|{\Tensor{X}{1}{n}{}}}$ is the inclusion mapping $\Tensor{X}{1}{n}{}\rightarrow \CTensor{X}{1}{n}{\alpha}$.  Assume first that $\Psi:\CTensor{X}{1}{n}{\beta}\longrightarrow \CTensor{X}{1}{n}{\alpha } $ is bounded. Let $r$ be fixed and let  $w,z \in \mathcal{S}^r_{X_1,\ldots,X_n}$. Then, $d_{\alpha}(w,z)=\|w-z\|_{{\alpha}}=\|\Psi(w-z)\|_{{\alpha}}\leq \|\Psi\| \cdot \|w-z\|_{{\beta}}=\|\Psi\|d_{\beta}(w,z)$. Consequently, the sequence is bounded by $\|\Psi\|$.

   To prove  the reciprocal statement, let $w,z \in \Tensor{X}{1}{n}{} $ and consider an upper bound $C$ of  the sequence. Since $w,z\in \mathcal{S}_{{\spaces{X}{1}{n}}}^r$ for some $r\in \mathbb{N}$, we have that   $\|w-z\|_{{\alpha}}=d_{\alpha}(w,z)\leq C d_{\beta}(w,z)=C \|w-z\|_{{\beta}}$ holds. Then,  the linear mapping  $\Psi_{|{\Tensor{X}{1}{n}{\beta}}}$ is   continuous. It can be extended to the completed space
  $  \CTensor{X}{1}{n}{\beta}$  preserving its norm, $C$.
  \end{proof}
Whenever  $X$ is infinite dimensional and $n>2$, the sequence $\{\sigma_{\spaces{X}{}{}}^r(\pi, \epsilon)\}_r$ is unbounded. This follows from John's result  (\cite{John1}, \cite{John2}) stating that  in such  case $X\hat{\otimes}_{\pi}\stackrel{n}{\cdots}\hat{\otimes}_{\pi}X \not\simeq  X\hat{\otimes}_{\epsilon}\stackrel{n}{\cdots}\hat{\otimes}_{\epsilon}X$.

     The case $n=2$ is different: In \cite{Groth},  A. Grothendieck wrote, "\underline{Comparaison {de} $E\hat{\otimes} F$ {et}  $E{\check{\otimes}} F$ }.\footnote{In our notation, $E\hat{\otimes}_{\pi} F$ and   $E{\hat{\otimes}_{\epsilon}} F, resp.$} Il est bien probable que si ces deux espaces sont identiques, $E$ ou $F$ est de dimension finie''.  Spaces  where $E\hat{\otimes}_{\pi} F \not\simeq E{\hat{\otimes}_{\epsilon}} F$ abound, but  they are not  all the infinite dimensional Banach spaces:    G. Pisier  in  \cite{Pisier},  provided examples of  infinite dimensional Banch spaces $X$ such that  $X\hat{\otimes}_\pi X \simeq X\hat{\otimes}_\epsilon X$.  To  prove  \cite[Theorem 3.2]{Pisier}, he shows  the existence of a constant $C>0$ such that, for every $k\in \mathbb{N}$ and for every $u\in S^k_{X, X}$, $\|u\|_{\pi}\leq C \|u\|_{\epsilon}$, which is   a bound of the sequence $\{\sigma_{X,X}^r(\pi, \epsilon)\}_r$.

\

\subsection{ On the completeness of ($S^r_{X_1,\ldots, X_n}, d_{\alpha})$.}

\

 The following table summarizes  the completeness properties of the set $S^r_{X_1,\ldots, X_n}$ according to rank and order,  in the case of infinite dimensional spaces. Recall that the completeness  of the set $S^r_{X_1,\ldots, X_n}$ (which  is the same as its  closedness in $\CTensor{X}{1}{n}{\alpha}$) is  independent of the reasonable cross  norm $\alpha$ considered on $\Tensor{X}{1}{n}{}$ (Corollary \ref{coro: completeness equiv all tensor norm}). The analogous  results for finite dimensional spaces   are  known;  their proofs   and their relation with the so-called Eckart-Young type approximation problems  can be found  in   \cite{DSilva Lim}.

\

\begin{center}
    \begin{tabular}{ | l |  l | p{8cm} |}
    rank  &  order &    \\ \hline
    $r=1$  & $n\geq 1 $ &   $ S^1_{X_1,\ldots, X_n}$,  the Segre cone $\Sigma_{X_1,\ldots, X_n}$,  is   closed.  %(Proposition \ref{prop: cSegre cerrado en tensor})
    \\ \hline
  $r\geq 1 $ & $n=2$  &
                                $S^r_{X_1, X_2}$ is  closed.
                               \\ \hline
    $r\geq 2$  & $n\geq 3$ &   $S^r_{X_1,\ldots, X_n}$  is not closed.  %(Proposition \ref{prop: Sr not closed}).
    \\
    \hline
    \end{tabular}
\end{center}

\

 The next lemmas collect some technical facts that  will be used  to prove the results in the table.   We  will use the notation $x_1\otimes \stackrel{{j}}{\stackrel{\vee}\cdots}\otimes x_n$  meaning $x_1\otimes\cdots \otimes x_{j-1}\otimes  x_{j+1}\otimes\cdots\otimes x_n$.

\begin{lemma}\label{lemma: linearly indep tensors} Given $n$ Banach  spaces  $\spaces{X}{1}{n}$,
for $j=\{1,\ldots, n\}$, let $ \{{x_j}^1, \ldots, x_j^{d_j}\}$ be a set of linearly independent vectors in $X_j$. Then,
\begin{enumerate}
\item[(i)]  The set $\{x^{k_1}_1\otimes\cdots\otimes x^{k_n}_n\}_{\substack{k_1=1,\ldots,d_1 \\\hspace{.4cm}\ldots \\  k_n=1,\ldots,d_n }}$ consist of linearly independent vectors in   $\Tensor{X}{1}{n}{}$.
\item[(ii)]  Let  $w=\sum_{i=1}^{r_w} w_i^1\otimes\cdots \otimes w_i^n \in \Tensor{X}{1}{n}{}$ be a  minimal decomposition, that is,  $Rank(w)=r_w$. Then,  for each $j\in\{1,\ldots,n\}$, the following set consists of linearly independent vectors:
\begin{align*} \{w_1^1\otimes \stackrel{{j}}{\stackrel{\vee}\cdots}\otimes w_1^n,\ldots,  w_{r_w}^1\otimes \stackrel{{j}}{\stackrel{\vee}\cdots}\otimes w_{r_w}^n\}\subset X_1\otimes\stackrel{{j}}{\stackrel{\vee}{\cdots}}\otimes X_n.
\end{align*}

\item[(iii)] For each $j=1,\ldots, n$,  let   $X_j=Y_j\oplus <z_0^j>$, with $0\neq z_0^j\in X_j$.   If  $x=y+z_0$, where  $y\in \Tensor{Y}{1}{n}{}$ and $z_0=z_0^1\otimes\cdots\otimes z_0^n$, then $Rank(x)=Rank(y)+1$.
 Also, the rank of $y$ is the same in $\Tensor{Y}{1}{n}{}$ than  in $\Tensor{X}{1}{n}{}$
\end{enumerate}
\end{lemma}
\begin{proof} The first and second assertions are  well known facts    which do not involve norms or continuity. They can be found in \cite[Chapter 3]{Land}. For the third one, we must keep in mind the continuity of the mappings involved.
For each $j=1,\ldots, n$, let  $\Pi_j\in \mathcal{L}(X_j, X_j)$ be   the bounded projection associated to  the decomposition   $X_j=\Pi_j(X_j)\oplus (I_j-\Pi_j)(X_j)=Y_j\oplus <z_0^j>$. They determine, in turn,
the following decomposition of the projective tensor product:
\begin{equation*}
\CTensor{X}{1}{n}{\pi}=(\CTensor{Y}{1}{n}{\pi})\oplus (<z_0^1>\hat{\otimes}_{\pi} Y_2\hat{\otimes}_{\pi}\cdots\hat{\otimes}_{\pi}  Y_n)\oplus \cdots \oplus (<z_0^1\hat{\otimes}_{\pi} \cdots \hat{\otimes}_{\pi}  z_0^n>).
\end{equation*}

Now, le  $x=y+z_0$ be an element as in  $(iii)$. It is clear that  that  $r_x\leq r_y+1$. We will see that the strict inequality never holds. To that end,  assume that $r_x\leq r_y$. Let  $x=\sum_{i=1}^{r_x} w_i^1\otimes\cdots \otimes w_i^n$.  Applying  to $x$ the   direct sum  decomposition stated above, we  get
$$
x= (\sum_{i=1}^{r_x}\Pi_1(w_i^1)\otimes \cdots \otimes \Pi_n(w_i^n)) \oplus (\sum_{i=1}^{r_x}(I_1-\Pi_1)(w_i^1){\otimes} \Pi_2(w_i^2)\otimes \cdots \otimes \Pi_n(w_i^n)) \oplus \cdots \oplus ( \lambda z_0^1\otimes \cdots\otimes z_0^n).
$$

Since $x=y+z_0$ is  also a decomposition  according to the same direct sum,  then necessarily  $\lambda =1$, every intermediate term is zero and $y=\sum_{i=1}^{r_x}\Pi_1(w_i^1)\otimes \cdots \otimes \Pi_n(w_i^n)$.  This gives a tensor decomposition of $y$ with $r_x\leq r_y$, thus, it must be $r_x=r_y$ and, consequently, (ii) holds for this expression of $y$.

Consider now the following  intermediate term of de decomposition of $x$:
\begin{align*} \sum_{i=1}^{r_y} (I_1-\Pi_{1})&(w_i^1)\otimes  \Pi_2(w_i^2)\otimes \cdots \otimes \Pi_n(w_i^n)=\sum_{i=1}^{r_y} (\lambda_i^1 z_0^1) \otimes\Pi_2(w_i^2)\otimes \cdots \otimes \Pi_n(w_i^n)= \\
                                                                & z_0^1  \otimes(\sum_{i=1}^{r_y} \lambda_i^1 \Pi_2(w_i^2)\otimes \cdots \otimes \Pi_n(w_i^n))=0.
\end{align*}

Then, $\sum_{i=1}^{r_y} \lambda_i^1 \Pi_2(w_i^2)\otimes \cdots \otimes \Pi_n(w_i^n)=0$.  This is a linear combination of linearly independent vectors  ((ii) for $j=1$), thus,
$\lambda_1^1=\ldots=\lambda_{r_y}^1=0$, that is, $w_i^1\in Ker(I_1-\Pi_1)=Y_1$, for every $i=1,\ldots, r_y$.
Arguing in an analogous way with the rest of the indices, $j=2,\ldots,n$, we get that $w_i^j\in Y_j$, for every $i=1,\ldots, r_y$. But this implies that $x=\sum_{i=1}^{r_y} w_i^1\otimes\cdots \otimes w_i^n\in \Tensor{Y}{1}{n}{}$, which is not possible, since $x=y+z_0$ and  $z_0\neq 0$.
This contradiction implies that $r_x>r_y$. Consequently,  $r_x=r_y +1$.
The assertion on the rank of $y$ is proved as follows:  let $R_X$ and $R_Y$ denote  the ranks of $y$ in $\Tensor{X}{1}{n}{}$ and $\Tensor{Y}{1}{n}{}$, respectively. Clearly, $R_X\leq R_Y$. To prove the other inequality, let     $y=\sum_{i=1}^{R_X}  x_i^1\otimes \cdots\otimes x_i^n$.  Since $y=(\Pi_1\otimes\cdots\otimes\Pi_n)(y)=\sum_{i=1}^{R_X} \Pi_1(x_i^1)\otimes \cdots \otimes \Pi_n(x_i^n)$, we have  a tensor decomposition of $y$ in $\Tensor{Y}{1}{n}{}$ with $R_X$ terms. This implies $R_Y\leq R_X$.
 \end{proof}

 \begin{lemma}\label{lemma: minimal rank} Let $z\in X_1\hat{\otimes}_{\epsilon}X_2$ and $r\in \mathbb{N}$. The minimal rank of $z$ satisfies  $r_{z}\leq r$ if and only if for every $\{\phi_1,\ldots ,\phi_{r+1}\}\subset X_1^*$, the set
 $\{(\phi_1\otimes Id)(z),\ldots ,(\phi_{r+1}\otimes Id)(z)\}$ is linearly dependent in $X_2$.
 \end{lemma}
 \begin{proof} We will use the natural  isometries  $X_1\hat{\otimes}_{\epsilon}X_2\hookrightarrow X_1^{**}\hat{\otimes}_{\epsilon}X_2\hookrightarrow  \mathcal{L}(X_1^*, X_2)$, $z\mapsto \delta_z,$ (see \cite[p. 46]{Ryan-libro}). The  rank as a tensor  of a  $z\in X_1\hat{\otimes}_{\epsilon}X_2$, $r_z,$  coincides with its rank as a linear mapping, $rk (\delta_z):= dim \{\delta_z(X_1^*)\}= dim \, span\{(\phi\otimes Id)(z); \phi \in X_1^*\}$. This is true also in  the case where both are infinite. By other hand, we have   that  $rk(\delta_z)\leq r$ if and only if  for every $\{\phi_1,\ldots ,\phi_{r+1}\}\subset X_1^*$, the set
 $\{(\phi_1\otimes Id)(z),\ldots ,(\phi_{r+1}\otimes Id)(z)\}$ is linearly dependent in $X_2$. Then, the lemma follows.
 \end{proof}

\begin{proposition}\label{prop: Sr Hilbert spaces} Let $X_1$ and $X_2$ be Banach spaces and let  $\alpha$  be a reasonable cross norm. Then, for every $r\geq 1$, the metric space $(S^r_{X_1, X_2}, d_{\alpha})$ is complete.
\end{proposition}
\begin{proof}   By  Corollary \ref{coro: completeness equiv all tensor norm}, it is enough to prove the result in the case where  $\alpha$ is the injective norm $\epsilon$. Let  $z\in \overline{S^r_{X_1, X_2}}$ and let  $(z_k)_k\subset S^r_{X_1, X_2}$ such that
$z_k\stackrel{\epsilon}{\rightarrow} z$.  For each $k$, let  $z_k=w_k^1\otimes y_k^1 + \ldots +w_k^r\otimes y_k^r$ be a tensor decomposition of $z_k$ (not necessarily minimal). Let $\{\phi_1,\ldots \phi_{r+1}\} \subset X_1^*$.  By the previous lemma, we know that the set $\{(\phi_1\otimes Id)(z_k),\ldots ,(\phi_{r+1}\otimes Id)(z_k)\}$ consists of linearly dependent vectors. Then, there exist scalars
$(a_{k,l})_{l=1}^{r+1}$  such that $\sum_{l=1}^{r+1}a_{k,l}(\phi_l\otimes Id)(z_k)=0.$ We can assume that the scalars are normalized, satisfying $\sum_{l=1}^{r+1}|a_{k,l}|^2=1$.  Then, $A_k:=(a_{k,1},\ldots,a_{k,r+1} ) \subset \mathbb{R}^{r+1}$ is a bounded set in a finite dimensional space which, necessarily,  has a subsequence that converges to some point $A=(a_{1},\ldots,a_{r+1} )$ with $\sum_{l=1}^{r+1}|a_l|^2=1$. Since $z_k\stackrel{\epsilon}{\rightarrow} z$, we have that $$0=\sum_{l=1}^{r+1}a_{k,l}(\phi_l\otimes Id)(z_k))\;\;{\longrightarrow}\;\;\sum_{l=1}^{r+1}a_{l}(\phi_l\otimes Id) (z).$$
Then,  $(\phi_1\otimes Id)(z),\ldots ,(\phi_{r+1}\otimes Id)(z)$ are linearly dependent vectors.  Using the previous lemma again, we get that $z$ has rank $\leq r$ , that is,  $z \in S^r_{X_1, X_2}$.
\end{proof}

 \begin{proposition}\label{prop: Sr not closed}   Let  $n>2$ and let      $\spaces{X}{1}{n}$ be Banach  spaces   of infinite  dimension. If  $\alpha$   is  a reasonable cross norm  on  $\Tensor{X}{1}{n}{}$ and    $r\geq 2$,   then the metric space  $(\mathcal{S}^r_{X_1,\ldots,X_n}, d_{\alpha})$ is not complete.
\end{proposition}

\begin{proof}  For each $i=1,\ldots, n$, consider a pair of linearly independent norm one vectors $z_i,w_i \in X_i$. For each $k\in \mathbb{N}$, let
$$x_k:=k(z_1+\frac{1}{k}w_1)\otimes\cdots\otimes(z_n+\frac{1}{k}w_n)-k (z_1\otimes\cdots\otimes z_n).$$
We have that $\{x_k\}_k\subset \mathcal{S}^2_{X_1,\ldots,X_n}$. This sequence converges  in $\CTensor{X}{1}{n}{\alpha}$ to
$$ x:= (w_1\otimes z_2 \otimes \cdots \otimes z_n) + (z_1\otimes w_2 \otimes z_3\otimes  \cdots \otimes z_n)+ \ldots + (z_1\otimes \cdots \otimes z_{n-1}\otimes w_n).  $$
 Consequently,    it is a $d_{\alpha}$-Cauchy sequence in $\mathcal{S}^2_{X_1,\ldots,X_n}$.
The proof will be done once we  check that the limit point $x$ is not an element of $\mathcal{S}^2_{X_1,\ldots,X_n}$.  We will use several times that  a set of the following type  (as well as any  of it subsets)
\begin{equation}\label{claim: lineraly independency}
\{ (z_1 \otimes \cdots \otimes z_n), (w_1\otimes z_2 \otimes \cdots \otimes z_n), (z_1\otimes w_2 \otimes \cdots \otimes z_n),  \ldots,(w_1\otimes \cdots \otimes w_n)\}
 \end{equation}
 consist of linearly independent vectors ((i) in Lemma \ref{lemma: linearly indep tensors}). This implies, in particular, that $x\neq 0$.
  If we assume  that $x\in \mathcal{S}^2_{X_1,\ldots,X_n}$, we have  two cases to check: when   $x$ has minimal  rank one, and when  $x$ has minimal rank two. If $x$ has minimal rank one,  then   $x=\Tensor{x}{1}{n}{}$ with $x_i\neq 0$. Consider any  $\phi\in (span\{x_1\})^{\perp}$.  Then
  \begin{align*}
(\phi\otimes Id)(x)=&\phi(w_1)z_2 \otimes \ldots \otimes z_n +   \cdots + \phi(z_1)z_2 \otimes\ldots \otimes  z_{n-1}\otimes w_n  =\phi(x_1^1)x_2^1\otimes\ldots \otimes x_n^1=0.
\end{align*}
 This implies that $\phi(w_1)=\phi(z_1)=0$. Consequently, $(span\{x_1\})^{\perp}\subset (span\{w_1\})^{\perp}\cap (span\{z_1\})^{\perp}$, which is not possible, since $(span\{x_1\})^{\perp}$ has codimension $1$  in $X_1^*$ while  $(span\{w_1\})^{\perp}\cap (span\{z_1\})^{\perp}$ has codimension $2$.

   In the case $x$ has minimal rank $2$, it admits a representation of the form $x=x_1\otimes\ldots \otimes x_n +y_1\otimes\ldots \otimes y_n$. Consider the two dimensional subspaces $Y_j:=span\{{x_j,y_j}\}\subset X_j,\, j=1,2$ and  their  annihilators:  $Y_j^{\perp}=\{\phi\in X_j^*;\ \phi_{|_{Y_j}}=0\}\subset X_j^*$, respectively. If $\phi\in Y_1^{\perp}$, then
 \begin{align*}
(\phi\otimes Id)(x)=&\phi(w_1)z_2 \otimes \ldots \otimes z_n +   \cdots + \phi(z_1)z_2 \otimes\ldots \otimes  z_{n-1}\otimes w_n \\ =&\phi(x_1)x_2\otimes\ldots \otimes x_n +\phi(y_1)y_2\otimes\ldots \otimes y_n=0.
\end{align*}
 This implies that $\phi(w_1)=\phi(z_1)=0$. Thus,   $\phi \in (span\{{w_1,z_1}\})^{\perp}$ and consequently  $Y_1^{\perp}\subset (span\{{w_1,z_1}\})^{\perp}$. This already implies that $Y_1=span\{{w_1,z_1}\}$.
 With an analogous argument, we get that $Y_2=span\{{w_2,z_2}\}$. For $j=1,2$,  let $\{\phi_j ,\psi_j\}\subset X_j^*$ be   {\sl biorthogonal} to $\{x_j,y_j\}\subset X_j$, respectively. Then,
   \begin{align*} (\phi_1\otimes \psi_2 \otimes Id)(x)&=\phi_1(w_1)\psi_2(z_2)z_3 \otimes \ldots \otimes z_n +  \cdots + \phi_1(z_1)\psi_2(z_2)z_3 \otimes\ldots \otimes  z_{n-1}\otimes w_n \\ & =\phi_1(x_1)\psi_2\cancel{(x_2)}x_3\otimes\ldots \otimes x_n +\phi_1\cancel{(y_1)}\psi_2(y_2)y_3 \otimes\ldots \otimes y_n=0.
 \end{align*}
  Because  $n>2$  and the vectors involving $z_i$'s and $w_i$'s are  linearly independent,  the coefficients are null.  This gives rise to the following system of equations:
 \begin{equation*}
\left.\begin{aligned}
\phi_1(w_1)\psi_2(z_2)=0 &\\
\phi_1(z_1)\psi_2(w_2)=0&\\
\phi_1(z_1)\psi_2(z_2)=0
\end{aligned}
\right\}
%\qquad \text{Maxwell's equations}
\end{equation*}
If $z_1\not\in Ker \, \phi_1$, we would have that $\psi_{|_{Y_2}}\equiv 0$ which is not possible, since $\psi_2(y_2)=1$. Then,   $z_1\in Ker \, \phi_1$ and, analogously,    $z_2\in Ker \, \psi_2$. Then, $span\{z_1\} = Ker \, \phi_1\cap Y_1=span\{y_1\}$, and $span\{z_2\}= Ker \, \psi_2\cap Y_2=span\{x_2\}$.

Computing  $(\psi_1\otimes \phi_2 \otimes Id)(x)$ and arguing in an analogous  way,  we get  that $span\{z_1\}= Ker \, \psi_1\cap Y_1=span\{x_1\}$ and $span\{z_2\} = Ker \, \phi_2\cap Y_2=span\{y_2\}$. This is in contradiction with  the fact that
 $dim \, (span\{x_1, y_1\})= dim \, (span\{x_2, y_2\}) =2$.
In this way, we obtain that $x$ does not admit a rank 2 expression.

To prove the remaining  cases, $r>2$, we use an induction argument on the rank of the tensors.  Assume that $(\mathcal{S}^{r-1}_{Y_1,\ldots,Y_n}, d_{\alpha})$ is not complete, for any infinite dimensional Banach spaces $\spaces{Y}{1}{n}$.  Let us fix a non null simple vector  $z_0:=z_0^1\otimes\cdots \otimes z_0^n\in \Tensor{X}{1}{n}{}$. For each $j=1,\ldots,n$, let $X_j=Y_j\oplus <z_0^j>$ be a decomposition as in (iii) in Lemmma \ref{lemma: linearly indep tensors}. We use the induction hypothesis to choose a sequence $\{y_k\}\subset\mathcal{S}^{r-1}_{Y_1,\ldots,Y_n}$ such that $\{y_k\}_k$ converges in norm to some $y\in \Tensor{Y}{1}{n}{}$, with $r-1<Rank(y)<\infty$.
By (iii) in Lemma \ref{lemma: linearly indep tensors}, we have that $\{y_k+z_0\}_k\subset\mathcal{S}^{r}_{X_1,\ldots,X_n}$.  The sequence $\{y_k+z_0\}_k$  converges in norm  to $x:=y+z_0$, whose rank, again by Lemma \ref{lemma: linearly indep tensors},  is $Rank(x)=Rank(y)+1>r$.
Consequently, the metric space $(\mathcal{S}^{r}_{X_1,\ldots,X_n}, d_{\alpha})$ is not complete.
\end{proof}

Proposition \ref{prop: Sr not closed} and its proof extend to the infinite dimensional case, the results in     Theorem \cite[4.10]{DSilva Lim}. There, the authors prove  that   whenever  the dimensions, $d_j<\infty$,  of the real spaces  are such that $2\leq r\leq \min(d_1,\ldots , d_n)$, the problem of determining a best
rank-r approximation for an order-k tensor in $\mathbb{R}^{d_1\times \cdots \times d_n}$ has no solution in general.

\subsection{The Segre cone.}

 If $\spaces{X}{1}{n}$ are  Banach  spaces    and  $\alpha$, $\beta$ is a pair of  reasonable cross norms  on  $\Tensor{X}{1}{n}{}$,   Theorem \ref{thm: fixed rank  cone metrics} establishes the
    Lipschitz equivalece:   $$(\Sigma_{\spaces{X}{1}{n}},  d_{\pi}) \hspace{.2cm}  \stackrel{Id}{\simeq}  \hspace{.2cm}  (\Sigma_{\spaces{X}{1}{n}},  d_{\alpha}) $$
    with
  $ d_{\alpha}(w,z)\leq d_{\pi}(w,z)\leq  2^{n-1} d_{\alpha}(w,z)$, for every $w,z \in \Sigma_{\spaces{X}{1}{n}}$.

   \begin{proposition}\label{prop: cSegre cerrado en tensor}  Let $\spaces{X}{1}{n}$ be Banach spaces.  $(\Sigma_{\spaces{X}{1}{n}}, d_{\pi})$ is a complete metric space.
     \end{proposition}

     \begin{proof} We will use induction in $n$. The result is trivial for $n=1$, since $\Sigma_{X_1}=X_1$.  Assume that the proposition holds  for any collection of $n-1$ Banach spaces. Let $\{x_j^1\otimes\cdots \otimes  x_j^n\}_j\subset \Sigma_{\spaces{X}{1}{n}}$ be  a sequence converging to $z\in \CTensor{X}{1}{n}{\pi}$.   To prove  that $z\in \Sigma_{\spaces{X}{1}{n}}$, we will consider two cases:

     \noindent
     \underline{Case 1}: There is some $k\in\{1,\ldots, n\}$ such that the sequences $\{x_j^k\}_j\subset  X_1$ and $\{x_j^1\otimes{\stackrel{\stackrel{k}{\vee}}{\cdots}} \otimes  x_j^n\}_j\subset X_1\hat{\otimes}_\pi {\stackrel{\stackrel{k}{\vee}}{\cdots}} \hat{\otimes}_\pi X_n$ have  no weakly null subsequences.
      For simplicity in notation, we will assume $k=1$. In this case, there exists a subsequence of indexes $\{j_k\}_k$ (that we continue to call $\{j\}_j$)
         and $\phi\in X_1^*$, $\psi\in (\CTensor{X}{2}{n}{\pi})^*$ such that $\phi(x_j) \xrightarrow[]{} \alpha_{\phi}\neq 0$ and $\psi(x_j^2\otimes\cdots \otimes  x_j^n) \xrightarrow[]{} \alpha_{\psi}\neq 0$.

         Let $\Pi_{\phi}: \CTensor{X}{1}{n}{\pi}\rightarrow \CTensor{X}{2}{n}{\pi}$ be  the linear mapping determined by the relation     $\Pi_{\phi}(y^1\otimes\cdots \otimes  y^n)=\phi(y^1)y^2\otimes\cdots \otimes  y^n$. Since $\Pi_{\phi}$ is continuous, we have that $\Pi_{\phi}(x_j^1\otimes\cdots \otimes  x_j^n)\xrightarrow[]{} \Pi_{\phi}(z)$. By other hand $\Pi_{\phi}(x_j^1\otimes\cdots \otimes  x_j^n)={\phi}(x_j^1)(x_j^2\otimes\cdots \otimes  x_j^n).$ Consequently $x_j^2\otimes\cdots \otimes  x_j^n \xrightarrow[]{\|\cdot\|} \frac{1}{\alpha_{\phi}}\Pi_{\phi}(z)$. By the induction hypothesis we know that
          $ \frac{1}{\alpha_{\phi}} \Pi_{\phi}(z)\in \Sigma_{\spaces{X}{2}{n}}. $
          With an analogous argument, defining  $\Pi_{\psi}: \CTensor{X}{1}{n}{\pi}\rightarrow X_1$ as $\Pi_{\psi}(y^1\otimes\cdots \otimes  y^n)=y^1 \psi(y^2 \otimes\cdots \otimes  y^n)$, we get that
          $x_j^1 \xrightarrow[]{\|\cdot\|} \frac{1}{\alpha_{\psi}}\Pi_{\psi}(z)\in X_1$.
          Then, we have that $x_j^1\otimes\cdots \otimes  x_j^n \xrightarrow[]{\|\cdot\|} z=\frac{1}{\alpha_{\psi}\alpha_{\phi}}\Pi_{\psi}(z)\otimes\Pi_{\phi}(z)\in \Sigma_{\spaces{X}{1}{n}}.$

       \noindent
        \underline{Case 2}: One of the bounded sequences,  $\{x_j^1\}_j\subset  X_1$  or $\{x_j^2\otimes{\stackrel{\stackrel{k}{\vee}}{\cdots}} \otimes  x_j^n\}_j$, has a weakly null subsequence.
         Let us assume  that   $\{x_j^1\}_j \xrightarrow[]{w} 0$.  Then, for every pair $\phi\in X_1^*$, $\psi\in (\CTensor{X}{2}{n}{\pi})^*$, $\lim{\phi}(x_j^1)\psi(x_j^2\otimes\cdots \otimes  x_j^n)=0$. Thanks to    Lemma 1.1 \cite{Lewis},  this is enough to ensure that   $x_j^1\otimes\cdots \otimes  x_j^n \xrightarrow[]{weak} 0$ in $ X_1\hat{\otimes}_\epsilon ( X_2\hat{\otimes}_\pi {\cdots} \hat{\otimes}_\pi X_n)$. Since the natural inclusion $ X_1\hat{\otimes}_\pi ( X_2\hat{\otimes}_\pi {\cdots} \hat{\otimes}_\pi X_n)\rightarrow X_1\hat{\otimes}_\epsilon ( X_2\hat{\otimes}_\pi {\cdots} \hat{\otimes}_\pi X_n)$ is continuous,  we have that  $x_j^1\otimes\cdots \otimes  x_j^n \xrightarrow[]{\|\cdot\|} z$ in $ X_1\hat{\otimes}_\epsilon ( X_2\hat{\otimes}_\pi {\cdots} \hat{\otimes}_\pi X_n)$.  Then necessarily $z=0,$ which is in $\Sigma_{\spaces{X}{1}{n}}$.  In case the weakly null sequence is  $\{x_j^2\otimes{\cdots} \otimes  x_j^n\}_j$, we can argue in a similar way. \end{proof}

\

     By  Corollary  \ref{coro: completeness equiv all tensor norm}, $\Sigma_{\spaces{X}{1}{n}}$  is closed in   $\CTensor{X}{1}{n}{\alpha}$ for any  cross norm $\alpha$.  This, along with Theorem \ref{thm: fixed rank  cone metrics}, allows us to define {\sl the Segre cone in the Banach space category}, by choosing one specific reasonable cross norm.  In this way, we introduce:

\begin{definition}\label{def: Sigma} The Segre cone $\Sigma_{\spaces{X}{1}{n}}$ of the $n$-Banach spaces $\spaces{X}{1}{n}$, is the metric space  $(\Sigma_{\spaces{X}{1}{n}}, d_{\pi})$.
 \end{definition} Observe that  as sets, the  Segre cone just defined  coincides with the algebraic Segre cone  introduced in  (\ref{diagram}).

\

     \begin{corollary}\label{coro: segre de sub cerrado}   Let $\spaces{Y}{1}{n}$ be Banach spaces and let  $X_i \subset Y_i$ be   closed subspaces, for every $i=1,\ldots,n$. Then
        $\Sigma_{\spaces{X}{1}{n}}$ is  closed in  $\Sigma_{\spaces{Y}{1}{n}}$.
     \end{corollary}
      \begin{proof} The continuous inclusion $\Tensor{X}{1}{n}{}\subset \CTensor{Y}{1}{n}{\pi}$ induces a reasonable cross norm $\alpha$ on   $\Tensor{X}{1}{n}{}$.  By Theorem \ref{thm: fixed rank  cone metrics}  and Proposition \ref{prop: cSegre cerrado en tensor}, we have that    $(\Sigma_{\spaces{X}{1}{n}},  d_{\alpha})$ is a complete space. This proves that $\Sigma_{\spaces{X}{1}{n}}$ is  closed in $\CTensor{Y}{1}{n}{\pi}$.  Using  Proposition \ref{prop: cSegre cerrado en tensor} again, it is also closed in $\Sigma_{\spaces{Y}{1}{n}}$.
     \end{proof}

      The following result is well known  in the algebraic finite dimensional setting:
     \begin{proposition}\label{prop: rulling subspaces} Let $\spaces{X}{1}{n}$ be Banach  spaces   and let $Y\subset \Sigma_{\spaces{X}{1}{n}}$ be a subspace of  $\dim{Y} \geq 2$. Then, there exists $i_0\in\{1,\ldots,n\}$ and $x_i \in X_i, i\neq i_0$  such that $Y\subset x_1\otimes\cdots \otimes x_{i_0-1}\otimes X_{i_0}\otimes x_{i_0+1}\otimes \cdots \otimes x_{n}$.
     \end{proposition}
     \begin{proof} It will be proved  by induction in $n$. The  proof of the case  $n=2$ in  \cite[Theorem 9.22]{Harris} can be easily adapted to the Banach space setting. For the sake of completeness, we include it. Observe that whenever $w_1\otimes w_2+z_1\otimes z_2\in \Sigma_{{X_1},{X_2}}$, then necessarily
     $<w_1>=<z_1>$ or $<w_2>=<z_2>$  ($<\cdot>$ denotes the span).  To see this, assume that $w_1\otimes w_2+z_1\otimes z_2= x_1\otimes x_2$. In this case, given $x^*\in <x_1>^{\perp}$  we have that $x^*(w_1)w_2=-x^*(z_1)z_2$. If $<w_2>=<z_2>$, the proof is done. In the other case, necessarily   $w_1, z_1 \in Ker \, x^*$.  This can be done with  any $x^*\in <x_1>^{\perp}$. Thus, $\cap_{x*\in <x_1>^{\perp} }Ker \, x^*=<x_1>=<w_1>=<z_1>$.

     Now, assume  we are in the case   $<w_1>=<z_1>$, and $\{w_2,z_2\}$ l.i. and consider an arbitrary  $y_1\otimes y_2\in Y$.  Arguing  as before, we  obtain that $<y_1>=<w_1>$. Consequently, $Y=w_1\otimes Y_2$. Analogously, when   $<w_2>=<z_2>$ we have that $Y=Y_1\otimes w _2$.

       Note that this result also holds  if it is  the case that the subspace  $Y\subset  \Sigma_{\spaces{X}{1}{n}}$  is contained in
     $ X_1\otimes w_2\otimes\cdots\otimes w_{n-1}\otimes  X_n$, for some  $w_i\in X_i$, $i=2,\ldots,n-1$.

       Now, let us assume that the result holds for any $n-1$ Banach spaces.  For a fixed not zero $x_1\in X_1$, let  $\phi$ be a projection onto the one dimensonal space $<x_1>\subset X_1$.  Consider the following bounded linear operators, where  $\phi(x)=\lambda(\phi)x_1$ for some $\lambda \in X_1^*$ and $\Psi$ is an isomorphims:
        \[ \begin{array}{ccccc}
        \CTensor{X}{1}{n}{\pi}&\stackrel{\phi\otimes Id}{\longrightarrow}&<x_1>\otimes\CTensor{X}{2}{n}{\pi}&\stackrel{\Psi }{\longrightarrow}&\CTensor{X}{2}{n}{\pi} \\
        x\otimes{x_2}\otimes\cdots\otimes x_n& \mapsto & \lambda(\phi)x_1\otimes{x_2}\otimes\cdots\otimes x_n & \mapsto & \lambda(\phi){x_2}\otimes\cdots\otimes x_n.
        \end{array}
        \]
       Let  $Y\subset \Sigma_{\spaces{X}{1}{n}}$ be a subspace.  Then  $\Psi\circ (\phi\otimes Id)(Y)$ is a subspace, which  is contained in $ \Sigma_{\spaces{X}{2}{n}}$.  By the induction hypothesis, we know that  there exist vectors $w_i, \,i=2,\ldots n-1$ and a subspace $Y_n\subset X_n$ such that
        $\Psi(\phi\otimes Id)(Y)=w_2\otimes\cdots\otimes w_{n-1}\otimes  Y_n$ (we have taken  the index $n$ for simplicity, but it could be any  between $2$ and $n$).  Then, $Y\subset X_1\otimes w_2\otimes\cdots\otimes w_{n-1}\otimes  Y_n.$  We  use now the  equivalent formulation of the  already proved  case $n=2$, to conclude that  $Y\subset w_1\otimes w_2\otimes\cdots\otimes w_{n-1}\otimes  Y_n$ or  $Y\subset X_1\otimes w_2\otimes\cdots\otimes w_{n-1}\otimes  w_n$ for some $w_1\in X_1$ or $w_n\in X_n$.
     \end{proof}

\section{$\Sigma$-operators.}\label{sect: morphisms}

 \subsection{Bounded multilinear mappings  as  Lipschitz mappings on the Segre cone}

 Having defined the Segre cone of $n$-Banach spaces, $\Sigma_{\spaces{X}{1}{n}}$, we are now able to define  $\Sigma$-operators. They are introduced to be used  as  a  geometric tool to study multilinear mappings in the Banach space category,  when applied through
  a scheme of factorization as in (\ref{diagram}) .  To this effect  we
 prove  in Theorem \ref{thm: continuity diagram} how the mappings involved in diagram (\ref{diagram}) behave with respect to the product of  norms, tensor metrics and
 reasonable cross norms, respectively.

   To fix notation, if    $T:\cartesian{X}{1}{n}\rightarrow Y$ is    a  multilinear mapping between vector spaces, we let
        $\hat{T} \in  {L}(\Tensor{X}{1}{n}{};  Y)$ be  the    unique  linear mapping satisfying  that  for every $x_i\in X_i \; i=1,\ldots, n $, \;
    $T(\spaces{x}{1}{n})=\hat{T}(\Tensor{x}{1}{n}{})$.

       \begin{definition}\label{def: sigma-oper} Given $n+1$ vector spaces $\spaces{X}{1}{n}, Y$, we will say that a mapping $f:\Sigma_{\spaces{X}{1}{n}}\rightarrow Y $ is a {\boldmath $\Sigma$}{\bf-operator} if there exist a multilinear mapping $T\in \mathcal{L}(\cartesian{X}{1}{n},Y)$ such that $f=\hat{T}_{|_{\Sigma_{\spaces{X}{1}{n}}}}$.
    We will denote:
     $${L}(\Sigma_{\spaces{X}{1}{n}}; Y)=\{f:\Sigma_{\spaces{X}{1}{n}}\rightarrow Y; \; f \, \mbox{is a}\, \Sigma\mbox{-operator}\} $$
 \end{definition}

  Diagram (\ref{diagram}) in the Introduction,   comprises  the relations between a multilinear mapping $T$,  its associated $\Sigma$-operator,  $f_T$,  and its associated  linear mapping $\hat{T}$.

  When $n=1$, the Segre cone is a Banach space,  $\Sigma_{X_1}=X_1$,  and  $T=f_T=\hat{T}$ are all linear.

If the spaces  have, in addition,  Banach space structures,  these mappings are simultaneously continuous, accordingly to  the following theorem:

\begin{theorem}\label{thm: continuity diagram}  Let  $\spaces{X}{1}{n}$  and $Y$  be Banach  spaces,  and let $T:\cartesian{X}{1}{n}\rightarrow Y$ be a multilinear mapping.  If we set $f_{T}:=\hat{T}_{|_{\Sigma_{\spaces{X}{1}{n}}}}$, then the following statements are equivalent:
  \item[i)] $T$ is bounded.
  \item[ii)]  $\hat{T} \in  \mathcal{L}(\CTensor{X}{1}{n}{\pi};  Y)$ is bounded.
   \item[iii)] $f_{T}:(\Sigma_{{\spaces{X}{1}{n}}},d_{\pi}) \rightarrow Y  $ is Lipschitz.

   In this case, $\|T\|=\|\hat{T}\|= \|f_{T}\|_{{Lip}}$.
\end{theorem}
\begin{proof} It is convenient to  recall that boundedness and continuity are equivalent notions for a linear or multilinear mapping (see \cite{DeFlo}, Chapter I).  Thus, the  equivalence between i) and ii) and the fact that $\|T\|=\|\hat{T}\|$ are,   essentially,    the universal property of the projective tensor product of Banach spaces (see, f.i. \cite[Theorem 1.1.8]{DieFouSwa}).  To prove the equivalence with iii), consider    $w=\Tensor{w}{1}{n}{}, z=\Tensor{z}{1}{n}{} \in \Sigma_{X_1,\ldots, X_n}$.
Then,  $\|f_{T}(w)- f_{T}(z)\|=\|\hat{T}(\Tensor{w}{1}{n}{})-\hat{T}(\Tensor{z}{1}{n}{})\|\leq \|\hat{T}\| \|w-z\|=\|\hat{T}\| d_{\pi}(w, z)$. This proves that ii) implies iii) and that $\|f_{T}\|_{{Lip}}\leq \|\hat{T}\|$.
 Finally, for an arbitrary $(w_1,\ldots,w_n)\in \cartesian{X}{1}{n}$, $\|T(w_1,\ldots,w_n)\|=\|f_{T}(w)-f_{T}(0)\| \leq  \|f_{T}\|_{{Lip}} \|w\|= \|f_{T}\|_{{Lip}} \|w_1\|\ldots\|w_n\| $, where $w=w_1\otimes\ldots \otimes w_n$. This proves that iii)  implies i) and that
 $\|T\|\leq \|f_{T}\|_{{Lip}}$.
\end{proof}

The subspace  of ${L}(\Sigma_{\spaces{X}{1}{n}}; Y)$  consisting of Lipschitz  $\Sigma$-operators will be denoted as $\mathcal{L}(\Sigma_{\spaces{X}{1}{n}}; Y)$. From Theorem \ref{thm: continuity diagram}, we have:

  \begin{proposition}\label{prop: Sigma morph}   $(\mathcal{L}(\Sigma_{\spaces{X}{1}{n}}; Y), \|\cdot \|_{Lip})$ is a Banach space isometrically isomorphic to both,  $\mathcal{L}(\spaces{X}{1}{n}; Y)$ and  $\mathcal{L}(\CTensor{X}{1}{n}{\pi}; Y)$.
   \end{proposition}

  The projective tensor norm is the only cross norm (up to isomorphisms) for which  the equivalence between $(ii)$ and $(iii)$ in Theorem \ref{thm: continuity diagram} holds:

\begin{proposition}\label{prop: corollary of main theorem}  Let  $\spaces{X}{1}{n}$  be Banach  spaces and let $\alpha$ be a cross norm on $\Tensor{X}{1}{n}{}$ such that any linear mapping $\hat{T}:\CTensor{X}{1}{n}{\alpha}\rightarrow Y$  is continuous if and only if its associated $\Sigma$-operator,  $f_{T}:(\Sigma_{{\spaces{X}{1}{n}}},d_{\alpha}) \rightarrow Y,   $ is Lipschitz. Then, necessarily  $\CTensor{X}{1}{n}{\alpha}\simeq \CTensor{X}{1}{n}{\pi}$.
\end{proposition}

\begin{proof} Assume that whenever $f_{T}$ is $d_{\alpha}$-continuous, then $\hat{T}$ is $\otimes_{\alpha}$-continuous. Let us consider the identity mapping on $\Tensor{X}{1}{n}{\pi}$. By Theorem \ref{thm: continuity diagram}, $f_{Id}:(\Sigma_{{\spaces{X}{1}{n}}},d_{\pi}) \rightarrow \Tensor{X}{1}{n}{\pi}  $ is Lipschitz and $\|f_{Id}\|_{{Lip}}=1$. If we use  Theorem \ref{thm: fixed rank  cone metrics} with the norms $\pi$ and $\alpha$, we  have that $d_{\pi}(w,z)\leq 2^n d_{\alpha}(w,z)$ for every $w, z \in \Sigma_{{\spaces{X}{1}{n}}}.$  Consequently $f_{Id}:(\Sigma_{{\spaces{X}{1}{n}}},d_{\alpha}) \rightarrow \Tensor{X}{1}{n}{\pi}  $ is $2^n$-Lipschitz.  Now, if we use the   hypothesis, we have that $\hat{Id}: \CTensor{X}{1}{n}{\alpha}\rightarrow\CTensor{X}{1}{n}{\pi}$  is continuous.  This already implies that   $\CTensor{X}{1}{n}{\alpha}\simeq \CTensor{X}{1}{n}{\pi}$. By Theorem \ref{thm: continuity diagram}, the reciprocal statement  also holds, .
\end{proof}

 In the case $\CTensor{X}{1}{n}{\alpha}\simeq \CTensor{X}{1}{n}{\pi}$,  the norms  $\|T\|, \|\hat{T}\|$ and $ \|f_{T}\|_{{Lip}}$ computed with $\alpha$,  being all finite, are not necessarily equal.

\begin{remark} \label{rmk:  frase a resaltar }  Despite the fact that the topology on $\Sigma_{{\spaces{X}{1}{n}}}$ induced by any  cross norm is unique, the weak topologies induced on it by $(\CTensor{X}{1}{n}{\alpha})^*$ are, in general, different.
To illustrate this, consider the sequence
$(e_i\otimes f_j)_{i,j}\subset \ell_2 \otimes \ell_2$, where  $(e_i)_i$ and $(f_j)_j$ are orthonortmal basis of $\ell_2$, respectively.  Since   $  (e_i\otimes f_j )_{i,j}$ is an orthonormal basis of $H_1\hat{\otimes}_H H_2$ (\cite[II.4 Proposition 1]{Reed Simon}), it  converges weakly to zero in $\ell_2\hat{\otimes}_{H}\ell_2$. However,  $(e_i\otimes f_i)_{i}$  is equivalent to the canonical basis of $\ell_1$ in  $\ell_2\hat{\otimes}_{\pi}\ell_2$ (see \cite[Corollary 5.14]{Holub}), which implies that it has no weakly convergent subsequences.

This fact  should be   be taken into account when defining ideals of multilinear operators:
As we will see, the notion of  $p$-summability  studied  in Definiton \ref{def: p-summ}  involves a  dual space, namely  $\mathcal{L}\left(\Sigma_{\spaces{X}{1}{n}}\right)$. In \cite{A FU} it is shown that, indeed, $p$-summability  depends on the  chosen tensor norm.  In what follows,  we  deal only with the projective norm.

\end{remark}

  \subsection*{From $\Sigma$-operators to   multilinear mappings: the  $p$-summability example.  }\label{sec: application}
    Based on Theorem \ref{thm: continuity diagram} we propose,  as a method to  generalize aspects  from the theory of bounded linear operators to the theory of multilinear  operators,  to do the following two steps:

    The first step is   to   replace boundedness conditions
on linear operators $\{S : X\rightarrow  Y \}$ by the analogous Lipschitz conditions on $\Sigma$-operators $\{f : \Sigma_{\spaces{X}{1}{n}} \rightarrow  Y \}$. The second step consists in formulate the results  obtained for $\Sigma$-operators, in multilinear terms, using the basic relation among them: $T(\spaces{x}{1}{n})=f_T(\Tensor{x}{1}{n}{}).$

   As an application of this method, we present the case of $p$-summability conditions, developed  in \cite{A FU}.    Further ideals are studied  in \cite{Samuel Thesis} with this approach.

   The notion of $p$-summing linear opeators   developed by  A. Pietsch  in  \cite{PietschSM67} (see \cite{DJT}),  has been generalized to different  settings, as is the case of  $p$-sumability in operator spaces \cite{Pisier Non-comm p-sum}, or   $p$-sumability in Lipschitz mappings \cite{FarmerJohnson}.
    In the case of    multilinear mappings, many  approaches to  $p$-summability   have  been appeared in the literature. A list of references on the subject may be found in \cite{A FU} and \cite{Calis Pelle}.

In this case, the first step of the general strategy we have just introduced  to deal with multilinear mappings is   the  following:

\begin{definition}\label{def: p-summ} \cite{A FU}
\label{defsumantes} Let  $X_1,\ldots, X_n, Y$ be Banach spaces
and $1\leq p <\infty$. A $\Sigma$-operator
$f\in\mathcal{L}(\Sigma_{X_1,\ldots,X_m}; Y)$ is said to be {\sl $p$-summing} if there is a
constant $c\geq 0$ such that for every  every $i=1,\ldots,k$
and  every  $u_i, v_i \in \Sigma_{\spaces{X}{1}{n}}$, the following inequality holds:
  \begin{equation*} \label{defsumantes.1}
  {\left(\sum_{i=1}^k \left\|{f}(u_i)-{f}(v_i)\right\|^p \right)^{1/p}  }
  \leq    c\cdot \sup \left\lbrace \left(\sum_{i=1}^k \left|\varphi(u_i)-\varphi(v_i)\right|^p\right)^{1/p}; \; \varphi\in B_{\mathcal{L}\left(\Sigma_{\spaces{X}{1}{n}}\right)}\right\rbrace
 \end{equation*}
 \end{definition}

 Pietsch' s original definition of absolutely $p$-summing  linear operators,  has every  $v_i=0$, \cite[Chapter 2]{DJT}.  Being  $f$ and $\varphi$   linear (thus,    $f(u_i)-f(v_i)=f(u_i-v_i)$ and $\varphi(u_i)-\varphi(v_i)=\varphi(u_i-v_i)$), note that  Definition \ref{def: p-summ} for the case    $n=1$, coincides with   Pietsch' s definition.

 \

   The main characteristic  features of absolutely  $p$-summing  operators remain true in the case of  $\Sigma$-operators. Among them,  we stressed  a  {\it Domination Theorem} and a  {\it Factorization Theorem}, see \cite{A FU}.

\subsubsection*{$p$-summability for multilinear operators}

 Returning to our original motivation, namely to develop a strategy to generalize linear ideals to the multilineal context, we obtain the following characterizations, directly  from the already mentioned results on $\Sigma$-operators:

\begin{theorem}\label{thm: equiv p summ multi op} Let $1 \leq p < \infty$ and  let $T:X_1\times\cdots\times X_n\rightarrow Y$ be an $n$-linear operator between Banach spaces. The following conditions for $T$ are equivalent:
\begin{enumerate}
\item There exists $c>0$ such that for $k\in\mathbb{N}$, $i=1,\ldots,k$ and $u_i, v_i \in \cartesian{X}{1}{n}$,
  \begin{equation*}
  {\left(\sum_{i=1}^k \left\|{T}\left(u_i\right)-T \left(v_i\right)\right\|^p \right)^{1/p}  }
  \leq    c\cdot \sup \left\lbrace \left(\sum_{i=1}^k \left|{{\varphi}}\left(u_i\right)-{\varphi}\left(v_i\right)\right|^p\right)^{1/p}; \; \varphi \in B_{\mathcal{L}\left(\cartesian{X}{1}{n}\right)}\right\rbrace
 \end{equation*}
 \item[]
\item There is a constant $c\geq0$ and a regular probability measure $\mu$ on $\left(B_{\mathcal{L}\left(X_1,\ldots,X_n\right)},w^*\right)$ such that for each $u=(u_1,\ldots,u_n), v=(v_1,\dots, v_n) \in \cartesian{X}{1}{n} $ we have that
\begin{equation*} \label{eqn: dominacion}
{\left\|T\left(u\right)-T\left(v\right)\right\| }
 \leq  c\cdot\left(\int_{B_{\mathcal{L}\left(X_1,\ldots,X_n\right)}} \hspace{-1.7cm}\left|\varphi\left(u\right)-\varphi\left(v\right)\right|^p\,d\mu(\varphi)\right)_{\ .}^{1/p}
\end{equation*}
\item[]

\item There exist  a  probability space $(\Omega,\Sigma,\mu)$, a  multilinear operator
operator $\nu: X_1\times\ldots \times X_n\rightarrow L_{\infty}(\mu)$ and   a Lipschitz
function ${\tilde{h_f}}: {L_p\left(\mu\right)}\rightarrow{\ell_{\infty}^{B_{Y^*}}}$
such that the following diagram commutes:
     \[
\begin{tikzcd}\label{p-summ multi  linfty}
        X_1\times\ldots \times X_n \arrow{r}{T}\arrow{dd}[swap]{\nu} &  Y \arrow{rd}{i_Y} & \\
        & & \ell_{\infty}^{B_{Y^*}} \\
        L_{\infty}(\mu)\arrow{r}{{i_p}}   &    L_p\left(\mu\right)\arrow{ru}{\tilde{h_f}} & \\
 \end{tikzcd}
\]

%  \arrow[phantom,"\cap"]{d} Pone una flecha fantasma, es decir, invisible, y en su lugar pone lo que esta entre las comillas.

\end{enumerate}

\end{theorem}
\qed

In the linear case, $n=1$,
Theorem \ref{thm: equiv p summ multi op} recovers the fundamental equivalent notions of $p$-summability for a linear operator, proved by A. Pietsch:   (1) is usually taken as the definition of a $p$-summing operator; (2) is known as the {\sl  Domination theorem of $p$-summing operators} and (3)  as  the {\sl Factorization Theorem of $p$-summing operators} (see \cite[2.12,2.13]{DJT}).

The validity of such fundamental equivalences in the more general context of multilinear mappings, makes them  a reasonable  proposal to generalize   $p$-summability to this setting. In this way, define:

  A multilinear operator $T\in\mathcal{L}(X_1,\ldots,X_m; Y)$ is {\bf Lipschitz  $p$-summing} if it satisfies any of the equivalent conditions in Theorem \ref{thm: equiv p summ multi op}, \cite[Definition 3.1]{A FU}.

\end{document}